\documentclass[11pt]{amsart}

\newif\ifarxiv
\arxivtrue

\newif\iftac
\tacfalse

\newcommand{\arxivonly}[1]{\ifarxiv#1\fi}
\newcommand{\taconly}[1]{\iftac#1\fi}
\newcommand{\arxivortac}[2]{\iftac#2\else#1\fi}

\usepackage[utf8]{inputenc}  

\iftac
\else
\usepackage{amsthm}
\fi
\usepackage{amsfonts}
\usepackage{amsmath}
\usepackage{amssymb}
\usepackage{mathtools}

\iftac
  \usepackage[square]{natbib}
\else
  \let\citep\cite
\fi

\iftac
\else
  \usepackage{xcolor}
  \definecolor{darkgreen}{rgb}{0,0.45,0}
  \definecolor{darkred}{rgb}{0.75,0,0}
  \definecolor{darkblue}{rgb}{0,0,0.6}
  \usepackage[colorlinks,citecolor=darkgreen,linkcolor=darkred,urlcolor=darkblue]{hyperref}
  \usepackage{breakurl}
\fi
\usepackage[mathscr]{eucal}

\iftac
\else
\calclayout
\fi

\usepackage[all]{xy}
\input{diagxy} 

\newbox\pbbox
\setbox\pbbox=\hbox{\xy \POS(50,0)\ar@{-} (0,0) \ar@{-} (50,50)\endxy}
\def\pb{\save[]+<3.8mm,-3.8mm>*{\copy\pbbox} \restore}
\iftac
\else
\theoremstyle{plain}
\newtheorem{theorem}{Theorem}

\newtheorem{lemma}[theorem]{Lemma}
\newtheorem{corollary}[theorem]{Corollary}
\fi

\renewcommand{\to}{\rightarrow} 
\newcommand{\into}{\hookrightarrow} 
\renewcommand{\equiv}{\simeq} 

\newcommand{\bbN}{\mathbb{N}}
\newcommand{\hattimes}{\mathbin{\widehat{\times}}}
\newcommand{\orth}{\mathrel{\pitchfork}}
\newcommand{\Id}{\mathtt{Id}}
\newcommand{\isProp}{\mathtt{isProp}\,}
\newcommand{\Prop}{\mathtt{Prop}}
\newcommand{\type}{\mathtt{Type}}
\newcommand{\El}{\mathtt{El}}
\newcommand{\U}{\mathtt{U}}

\iftac
\usepackage[colorlinks=true]{hyperref}
\hypersetup{allcolors=[rgb]{0.1,0.1,0.4}}
\fi

\begin{document}
\title{The law of excluded middle in the \\ simplicial model of type theory}
\author{Chris Kapulkin \arxivortac{\and}{and} Peter LeFanu Lumsdaine}
\iftac
\else
\date{September 2, 2020}
\fi

\taconly{\maketitle}

\begin{abstract}
  We show that the law of excluded middle holds in Voevodsky’s simplicial model of type theory.
  As a corollary, excluded middle is compatible with univalence.
\end{abstract}

\arxivonly{\maketitle}

\taconly{\medskip}

Since \citep{kapulkin-lumsdaine:simplicial-model} first appeared in 2012, various readers have wondered whether Voevodsky’s model of type theory in simplicial sets validates the law of excluded middle.
This fact is by now folklore within the field (implicitly appealed to in \citep[\textsection 3.4]{hott:book}, for instance, for the relative consistency of LEM); but since it has still not appeared in the literature, we set it down here for the record.

We assume \citep{kapulkin-lumsdaine:simplicial-model} as background throughout, and follow its notational conventions, with a few shorthands for readability: we omit Scott brackets, write $\Gamma \models A\ \type$ to mean that $A$ is a type of the simplicial model (i.e.,\ a Kan fibration $p_A : \Gamma.A \to \Gamma$), and write $\Gamma \models A$ to mean that $p_A$ admits a section, i.e., $A$ is inhabited.

As required for constructing the simplicial model as in \citep[Cor.~2.3.5]{kapulkin-lumsdaine:simplicial-model}, we assume throughout an inaccessible cardinal $\alpha$, and later another $\beta < \alpha$ to give a universe $U_\beta$ in the model.

For $\Gamma \models A \ \type$, define
$\isProp A \coloneqq \prod_{x, y : A} \Id_A(x,y)$.
Our main goal is:

\begin{theorem}[Schema of Excluded Middle] \label{thm:lem}
  Let $\Gamma \models A \ \type$, and suppose $\Gamma \models \isProp A$.
  Then $\Gamma \models A + \neg A$.
\end{theorem}

We write $i_n \colon \partial \Delta^n \into \Delta^n$ for the boundary inclusion of the standard $n$-simplex, and $f \orth g$ to indicate that $f$ has the left lifting property with respect to $g$.

\begin{lemma} \label{lem:orthogonality-combinatorics}
  The following are equivalent for a Kan fibration $p$:
  \begin{enumerate}
    \item $i_1 \hattimes i_n \orth p$ for all $n \geq 0$;
    \item $i_n \orth p$ for all $n \geq 1$.
  \end{enumerate}
\end{lemma}

\begin{proof}
  Standard combinatorics of prisms, similar to~\citep[proof of Thm.~1.5.3]{joyal-tierney:simplicial-homotopy}.
\end{proof}

\begin{lemma} \label{lem:break-fibration}
  Given a Kan fibration $p \colon Y \to X$, the image of $p$ is complemented: that is, the sets $\{X_n \setminus p(Y_n)\}_{n \in \bbN}$ form a simplicial set $X \setminus p(X) \subseteq X$.
\end{lemma}

\begin{proof}
  For any $n$-simplex $x \in X_n$, note that $x \in X_n \setminus p(Y_n)$ exactly when all vertices of $x$ lie in $X_0 \setminus p(Y_0)$. The claim follows directly.
\end{proof}

\begin{proof}[\arxivonly{Proof }of Theorem~\ref{thm:lem}]
  Suppose $\Gamma \models \isProp A$.
  Unwinding the interpretation of $\isProp$ in the simplicial model, this says that the two projections $\pi_1, \pi_2 \colon \Gamma.A.A \to \Gamma. A$ are homotopic over $\Gamma$; equivalently, the fibration $p_{\Id_A} \colon \Gamma. A . A. \Id_A \to \Gamma. A. A$ is trivial.
  But $p_{\Id_A}$ is a pullback of the Leibniz exponential $i_1 \triangleright p_A$ along a weak equivalence, so the latter is also trivial:
  \[ \xymatrix@C=1em{
      \Gamma.A.A.\Id_A \ar[r]^-{\sim}  \ar[d]_{p_{\Id_A}} \pb & (\Gamma.A)^{\Delta^1} \ar[d]^{i_1 \triangleright p_A}
      \\ \Gamma.A.A \ar[r]^-{\sim} \ar[d] \pb & (\Gamma.A)^2 \times_{\Gamma^2} \Gamma^{\Delta^1} \ar[d]
      \\ \Gamma \ar[r]^-{\sim} & \Gamma^{\Delta^1}
    } \]
  This is in turn equivalent to $i_1 \hattimes i_n \orth p_A$ for all $n$; so by Lemma~\ref{lem:orthogonality-combinatorics}, $i_n \orth p_A$ for all $n \geq 1$.

  Now to give a section of $p_{A + \neg A}$, we decompose $\Gamma$ according to Lemma~\ref{lem:break-fibration} as $\Gamma = \Gamma_0 + \Gamma_1$, where $\Gamma_0 = p_A (\Gamma.A)$ and $\Gamma_1 = \Gamma \setminus \Gamma_0$, and work over each component separately.
  The pullback of $p_A$ to $\Gamma_0$ is orthogonal to $i_0$ by definition of $\Gamma_0$, and higher $i_n$ since $p_A$ was; so it is a trivial fibration, so admits a section.
  Over $\Gamma_1$, the pullback of $p_A$ is empty, so we have a section of $p_{\neg A}$.
  Together they give the desired section $\Gamma \to \Gamma.A + \neg A$ of $p_{A + \neg A}$.
\end{proof}

Theorem~\ref{thm:lem} gave the law of excluded middle in the form of a global scheme.
This immediately implies other forms of LEM, e.g.\ quantified over an universe as in \citep[(3.4.1)]{hott:book}.
Let $\U_\beta$ be a universe in the model, and define $\Prop_\beta \coloneqq \sum_{A : \U_\beta} \isProp A$.

\begin{corollary} \label{cor:lem-in-universe}
  The universe $\U_\beta$ satisfies LEM: that is,
  \[ \textstyle \models \prod_{A : \Prop_{\beta}} \left( \El(\pi_1(A)) + \neg \El(\pi_1(A))\right). \]
\end{corollary}

\begin{proof}
 Apply Theorem~\ref{thm:lem} to the type ${A : \Prop_{\beta}} \models \El(\pi_1(A))\ \type$. 
\end{proof}

\begin{corollary}
  It is consistent, over Martin-Löf Type Theory with $\Pi$-, $\Sigma$-, $\Id$-, $1$-, $0$-, and $+$-types (as set out in \citep[App.~A, B]{kapulkin-lumsdaine:simplicial-model}), for a universe to simultaneously satisfy the univalence axiom, the law of excluded middle, and closure under all the listed type formers. 
\end{corollary}

\begin{proof}
  By Corollary~\ref{cor:lem-in-universe} together with \citep[Cor.~2.3.5]{kapulkin-lumsdaine:simplicial-model}.
\end{proof}

\begin{corollary}
  In each simplicial universe $\beta$, the type of propositions is equivalent to a discrete simplicial set with 2 elements, i.e., $\Prop_\beta \equiv 1 + 1$.
\end{corollary}

\begin{proof}
  This follows internally from Corollary~\ref{cor:lem-in-universe}, by \citep[Ex.~3.9]{hott:book}.
\end{proof}

\subsection*{Acknowledgements}
We are grateful to Christian Sattler for catching an error in an earlier version of this paper.

\iftac
\else
\bibliographystyle{amsalphaurlmod}
\bibliography{general-bibliography}
\fi

\end{document}